\documentclass[11pt]{amsart}
\usepackage{amsmath, amstext, amsbsy, amssymb}

\hoffset \voffset \oddsidemargin=55pt \evensidemargin=55pt
\numberwithin{equation}{section}
\def\beq{\begin{eqnarray}}
\def\eeq{\end{eqnarray}}
\def\beqs{\begin{eqnarray*}}
\def\eeqs{\end{eqnarray*}}
\def\NN{{\mathbb N}}
\def\mz{{\mathbb Z}}
\def\mr{{\mathbb R}}

\def\ind{\hbox{\rm ind}}
\def\ord{\hbox{\rm ord}}

\newfont{\df}{eufm10}

\def\mod{{\hbox{\rm mod}}}

\title[Index of minimal zero-sum sequences]
{ On the index of length four minimal zero-sum sequences}
\thanks{$^\dag$The corresponding author's email: xialimeng@ujs.edu.cn}
\author[C.-X. Shen]{Caixia   Shen$^1$}
\author[L.-M. Xia]{Li-meng   Xia$^{\dag,1}$}
\author[Y.-L. Li]{Yuanlin Li$^2$}

\date{}
\begin{document}
\maketitle
\centerline{$^1$Faculty of Science, Jiangsu University, Zhenjiang, 212013, Jiangsu Pro., China}
\centerline{$^2$Department of Mathematics, Brock University, St. Catharines, ON, Canada L2S 3A1}

\def\abstractname{ABSTRACT}
\begin{abstract}
Let $G$ be a finite cyclic group. Every sequence $S$ over $G$ can be written in the form $S=(n_1g)\cdot\ldots\cdot(n_lg)$ where $g\in G$ and $n_1, \ldots, n_l\in[1, \ord(g)]$, and the index $\ind(S)$ of $S$ is defined to be the minimum of $(n_1+\cdots+n_l)/\ord(g)$ over all possible $g\in G$ such that $\langle g \rangle =G$.
 A conjecture on the index of length four sequences says that every minimal zero-sum sequence of length 4 over a finite cyclic group $G$ with $\gcd(|G|, 6)=1$ has index 1. The conjecture was confirmed recently for the case when $|G|$ is a product of at most two prime powers. However, the general case is still open. In this paper, we make some progress towards solving the general case.  Based on earlier work on this problem, we show that  if $G=\langle g\rangle$ is a  finite cyclic group of order $|G|=n$ such that $\gcd(n,6)=1$  and $S=(x_1g)(x_2g)(x_3g)(x_4g)$ is  a minimal zero-sum  sequence over $G$ such that $x_1,\cdots,x_4\in[1,n-1]$ with $\gcd(n,x_1,x_2,x_3,x_4)=1$, and $\gcd(n,x_i)>1$ for some $i\in[1,4]$, then  $\ind(S)=1$.
 By using an innovative method developed in this paper, we are able to give a new (and much shorter) proof to the index conjecture for the case when $|G|$ is a product of two prime powers.


\end{abstract}

\newtheorem{theo}{Theorem}[section]
\newtheorem{theorem}[theo]{Theorem}
\newtheorem{defi}[theo]{Definition}
\newtheorem{conj}[theo]{Conjecture}
\newtheorem{lemma}[theo]{Lemma}
\newtheorem{coro}[theo]{Corollary}
\newtheorem{proposition}[theo]{Proposition}
\newtheorem{remark}[theo]{Remark}

\setcounter{section}{0}

\section{Introduction}
Throughout the paper, let $G$ be an additively written finite cyclic group of order $|G| = n$. By
a sequence over $G$ we mean a finite sequence of terms from $G$ which is unordered and repetition
of terms is allowed. We view sequences over $G$ as elements of the free abelian monoid $\mathcal{F}(G)$
and use multiplicative notation. Thus a sequence $S$ of length $|S| = k$ is written in the form
$S = (n_1g)\cdot...\cdot(n_kg)$, where $n_1,\cdots,n_k\in{\mathbb N}$ and $g\in G$. We call $S$ a {\it zero-sum sequence} if $\sum^k_{j=1}n_jg = 0$. If $S$ is a zero-sum sequence, but no proper nontrivial subsequence of $S$ has sum zero, then $S$ is called a {\it minimal zero-sum sequence}. Recall that the index of a sequence $S$ over $G$
is defined as follows.

\begin{defi}
For a sequence over $G$
\beqs S=(n_1g)\cdot...\cdot(n_kg), &&\hbox{where}\;1\leq n_1,\cdots,n_k\leq n,\eeqs
 the index of $S$ is defined  by $\ind(S)=\min\{\|S\|_g|g\in G \hbox{~with~}\langle g\rangle=G\}$, where
\beq \|S\|_g=\frac{n_1+\cdots+n_k}{\ord(g)}.\eeq
\end{defi}
\noindent Clearly, $S$ has sum zero if and only if $\ind(S)$ is an integer. We note that there are also slightly different definitions of the index in the literature, but they are all equivalent (see \cite[Lemma 5.1.2]{Ger2}).

\begin{conj}
Let $G$ be a finite cyclic group such that $\gcd(|G|,6)=1$. Then every minimal zero-sum sequence $S$ over $G$ of length $|S|=4$ has $\ind(S)=1$.
\end{conj}

If $S$ is a minimal zero-sum sequence of length $|S|$ such that $|S|\leq3$ or $|S|\geq\lfloor \frac{n}2\rfloor+2$, then $\ind(S)=1$ (see [12, 14]). In contrast to that, it was shown that for each $k$ with $5\leq k\leq \lfloor \frac{n}2\rfloor+1$, there is a minimal zero-sum subsequence $T$ of length $|T| = k$ with $\ind(T)\geq 2$ ([11, 13]) and that the same is true for $k = 4$ and $\gcd(n, 6)\not= 1$ ([11]). The only unsolved case leads to the above conjecture.

In [10], it was proved that Conjecture 1.2 holds true if $n$ is a prime power. Recently in [9], it was proved that Conjecture 1.2 holds for $n=p_1^\alpha\cdot p_2^\beta$ (a product of two prime powers) with the restriction that at least one $n_i$ is co-prime to $|G|$. In a most recent paper [18], the conjecture was confirmed for the remaining situation in the case when $n=p_1^\alpha\cdot p_2^\beta$. Thus these two papers together completely settle the case when $n$ is a product of two prime powers.

Let $S =(n_1g)\cdot...\cdot(n_kg)$ be a minimal zero-sum sequence over $G$. Then $S$ is called reduced if $(pn_1g)\cdot...\cdot (pn_kg)$ is no longer a minimal zero-sum sequence for every  prime factor $p$ of $n$. In [17] and [19], Conjecture 1.2 was proved if the sequence $S$ is reduced. However, the general case is still open. In the present paper, we make some progress towards solving the general case and obtain the following main result.

\begin{theo}\label{mainthm}
Let $G=\langle g\rangle$ be a  finite cyclic group of order $|G|=n$ such that $\gcd(n,6)=1$. Let $S=(x_1g)(x_2g)(x_3g)(x_4g)$ be a minimal zero-sum  sequence over $G$, where $g\in G$ with $\ord(g)=n$ and $x_1,\cdots,x_4\in[1,n-1]$ with $\gcd(n,x_1,x_2,x_3,x_4)=1$, and $\gcd(n,x_i)>1$ for some $i\in[1,4]$. Then  $\ind(S)=1$.
\end{theo}

\section{Preliminaries}
Recall that  $G$ always denotes a finite cyclic group of order $|G|=n$. Given real numbers $a, b\in\mr$, we use $[a,b]=\{x\in\mz| a\leq x\leq b\}$ to denote the set of integers between $a$ and $b$.  For $x\in\mz$, we denote by $|x|_n\in[1,n]$ the integer congruent to $x$ modulo $n$.  Let $S=(x_1g)(x_2g)(x_3g)(x_4g)$ be a minimal zero-sum sequence over $G$ such that $\ord(g)=n=|G|$ and $1\leq x_1,x_2,x_3,x_4\leq n-1$. For convenience, we set $f(x_i):=\gcd(n,x_i)$ for $i\in[1,4]$. In what follows we always assume that $\gcd(n,x_1,x_2,x_3,x_4)=1$,  so we have $\gcd(f(x_i),f(x_j),f(x_k))=1$ for any three different $i,j,k$. 
The following lemma is crucial and will be used frequently in sequel.

According to the assumption of Theorem \ref{mainthm}, the order $n$ of group $G$ is not a prime number (since $1<\gcd(n,x_i)\leq n-1<n$ for some $i\in[1,4]$).
In what follows, we may always assume that $n$ is an arbitrary positive integer such that $\gcd(n,6)=1$ and $n$ is not a prime number unless state otherwise.

\begin{lemma} \label{L2.1}\cite[Remark 2.1]{LP2}\\
(1)  If there exits a positive integer $m$ such that $\gcd(n,m)=1$ and at most one $|mx_i|<\frac{n}{2}$ (or, similarly, at most one $|mx_i|>\frac{n}{2}$), then $\ind(S)=1$.\\
(2) If there exits a positive integer $m$ such that $\gcd(n,m)=1$ and $|mx_1|_n+|mx_2|_n+|mx_3|_n+|mx_4|_n=3n$, then $\ind(S)=1$.
\end{lemma}

Denote by $U(n)$ the unit group of $n$, i.e.  $U(n) =\{k \in \NN |1\leq k\leq n-1, \gcd(k, n)=1 \}$. Thus $|U(n)|=\varphi(n)$ where $\varphi$ is the {\it Euler $\varphi$-function}.
We note that for any $y \in U(n)$ $\ind(S)=\ind(yS)$ where $yS=(|yx_1|_ng)(|yx_2|_ng)(|yx_3|_ng)(|yx_4|_ng)$.

\begin{lemma}\label{L2.2}
Let $p$ be a prime factor of $n$ and  $\alpha=\frac{n}{p}$. 
Then for any $1\leq v<n$ there exist elements $1+k\alpha, 1+j\alpha\in U(n)$ such that $|v+k\alpha|_n<\frac{n}{2}$ and $|v+j\alpha|_n>\frac{n}{2}$. Moreover, if $\gcd(v,p)=1$, then there exists $y=1+t\alpha\in U(n)$ such that $|yv|_n<\frac{n}{2}$.
\end{lemma}
\begin{proof}
If $y=1+t\alpha\not\in U(n)$, then there exists prime factor $q|\gcd(n,y)$. If $q\not=p$, we have $q|\alpha$, and thus $q|\gcd(y,\alpha)=1$, a contradiction.  We infer that $p|y$ and $\gcd(p,\alpha)=1$. It is easy to check that at most one $t<p$ such that $y=1+t\alpha\not\in U(n)$. So we may assume that for some $t_0$, all $p-1$ terms $|1+t_0\alpha|_n, |1+(t_0+1)\alpha|_n, \dots, |1+(t_0+p-2)\alpha|_n$ are in $U(n)$. If  all the corresponding terms $|v+t\alpha|_n$ with $t_0\leq t\leq t_0+p-2$ stand in the same side of $\frac{n}{2}$,  then without loss of generality, we may assume that all these terms $|v+t\alpha|_n <\frac{n}{2} $, where $t_0\leq t\leq t_0+p-2$. Since $(v+(t+1)\alpha)-(v+t\alpha)=\alpha < n/4$ ($t_0\leq t \leq p-2$), we conclude that any two consecutive terms $(v+(t+1)\alpha)$ and $(v+t\alpha)$  fall into the same interval $[n\lfloor\frac{v+t\alpha}{n}\rfloor,  n\lfloor\frac{v+t\alpha}{n}\rfloor +\frac{n}{2}]$. Thus all the above terms fall into the same interval, so we have $ b=v+t_0\alpha<v+(t_0+1)\alpha<\cdots<v+(t_0+p-2)\alpha<b+\frac{n}{2}$. Hence we infer that $(p-2)\alpha<\frac{n}{2}$, which implies that $p<4$, giving a contradiction as $\gcd(n,6)=1$ and $p|n$. Thus the first statement holds.

Next assume that $\gcd(v, p)=1$. We note that if $0\leq t_1\not= t_2\leq p-1$, then $|v(1+t_1\alpha)|_n\not= |v(1+t_2\alpha)|_n$. Thus as a set $\{|v|_n, |v(1+\alpha)|_n, \dots, |v(1+(p-1)\alpha)|_n\}= \{|v|_n, |v+\alpha|_n, \dots, |v+(p-1)\alpha|_n\}$. As above, we can prove that there exists $y=1+t\alpha \in U(n)$ such that $|yv|_n < \frac{n}{2}$.
\end{proof}

\begin{remark}\label{RM} We note that if $p^2|n$, then $y=1+t\alpha \in U(n)$ for any $t \in [0,p-1] $. If $p|n$ and  $p^2\not|n$, then $\gcd(p,\alpha)=1$, and so there is a unique $t\in[0,p-1]$ such that $y=1+t\alpha\not\in U(n)$. In particular, if $v\in[1,n-1]$ and $p|v$, then $|yv|_n=v$ for any $y=1+t\alpha$.
\end{remark}

\begin{coro} \label{C2.3}
If $p^s|\beta < n$, $p^{s+1}\not\!|\beta$ and $p^{s+1}|n$, then there exists $y=1+\frac{tn}{p^{s+1}}\in U(n)$ (with $0\leq t < p$) such that $|y\beta|_n<\frac{n}{2}$.
\end{coro}


\begin{proof}
Let $\beta_1 =\frac{\beta}{p^s}, n_1=\frac{n}{p^s}$ and $\alpha =\frac{n_1}{p}=\frac{n}{p^{s+1}}$. Then $1\leq \beta_1 < n_1$ and $\gcd(\beta_1, p)=1$. By Lemma~\ref{L2.2}, there exists $y=1+t\alpha \in U(n_1) \subseteq U(n)$ such that $|y\beta_1|_{n_1} <\frac{n_1}{2}$. Thus $|y\beta|_n=|y\beta_1p^s|_n= p^s|y\beta_1|_{n_1}<p^s\frac{n_1}{2}=\frac{n}{2}$ as desired.
\end{proof}

\begin{lemma} \label{L2.4}
If $f(x_1)=f(x_2)=d>1$, then $\ind(S)=1$.
\end{lemma}
\begin{proof}
We first show that there exists $u\in U(n)$ such that $|ux_1|_n<\frac{n}{2}$ and  $|ux_2|_n<\frac{n}{2}$. By multiplying $S$ by a unit, we may assume that $x_1=d$ and $x_2=n-kd$, where $k\in U(n)$. If $kd>\frac{n}{2}$, then we are done. So  we may assume that  $kd<\frac{n}{2}$. Since $S$ is a minimal zero-sum sequence, we conclude that $k\not=1$, so $x_1=d<\frac{n}{2k}\leq\frac{n}{4}$.  If $kd>\frac{n}{4}$, then $2x_1=2d\leq kd<\frac{n}{2}$ and $\frac{n}{2}<2kd<n$.
Let $u=2$. Then we get $|ux_1|_n<\frac{n}{2}$ and $|ux_2|_n<\frac{n}{2}$ as desired. If $kd<\frac{n}{4}$, then there exists $s$ such that $2^{s}x_1<\frac{n}{4}\leq 2^skd<\frac{n}{2}$. Let $u=2^{s+1}$. Then $|ux_1|_n<\frac{n}{2}$ and $|ux_2|_n<\frac{n}{2}$ as desired.

Next we  may assume that $x_1 <\frac{n}{2}$ and $x_2<\frac{n}{2}$.  Let $p$ be a prime factor of $d$ and $\alpha=\frac{n}{p}$. Then $\gcd(p,x_3)=1$. By Lemma~\ref{L2.2}, there exists $y=1+j\alpha\in U(n)$ such that $|yx_3|_n<\frac{n}{2}$. Since $y$ fixes $x_1$ and $x_2$ (i.e. $|yx_1|_n=x_1$ and $|yx_2|_n=x_2$),   by (1) of Lemma~\ref{L2.1}, we have $\ind(S)=\ind(yS)=1$.
\end{proof}

Next we assume that $n$ has at least three prime factors. Then for every prime $p|n$, we have $p\geq 11$ or $\alpha=\frac{n}{p}\geq 55$.  This estimate for $\alpha$ will be used in Lemmas \ref{L2.5}-\ref{L2.6}, and then in Lemmas \ref{L2.8}-\ref{L2.9}.
\begin{lemma} \label{L2.5}
If $f(x_1)=7$, $\gcd(f(x_1),f(x_2))=\gcd(f(x_1),f(x_3))=\gcd(f(x_1),f(x_4))=1$ and $7^2\not|n$, then $\ind(S)=1$.
\end{lemma}

\begin{proof}
Let $\alpha=\frac{n}{7}$. As noted in  Remark~\ref{RM} there exist exactly six $t\in[0,6]$ such that $y=1+t\alpha\in U(n)$. By multiplying $S$ with a suitable unit, we may assume that $x_1=\frac{n-7}{2}$. 
Note that $|yx_1|_n=x_1 < \frac{n}{2}$ for any $y=1+t\alpha\in U(n)$. We may also assume that exactly one of $|yx_2|_n, |yx_3|_n, |yx_4|_n$ is less than $\frac{n}{2}$. For otherwise, it follows from Lemma~\ref{L2.1} that $\ind(S)=1$, and we are done.

We claim that there exist at most two elements $y=1+t\alpha\in U(n)$  such that both $|yx_3|_n>\frac{n}{2}$ and $|yx_4|_n>\frac{n}{2}$. For otherwise, we infer that either at least five $|yx_3|_n$ or at least five $|y'x_4|_n$ are greater than $\frac{n}{2}$. As in the proof of Lemma~\ref{L2.2}, this implies that $(5-1)\alpha < \frac{n}{2}$, so $\frac{4}{7}n < \frac{n}{2}$, giving a contradiction.

If there exists at most one element $y=1+t\alpha\in U(n)$ such that $|yx_3|_n>\frac{n}{2}$ and $|yx_4|_n>\frac{n}{2}$, then there exist at least five elements $y=1+t\alpha\in U(n)$ such that $|yx_3|_n$ and $|yx_4|_n$ stand in different sides of $\frac{n}{2}$. Hence by the assumption that exactly one of $|yx_2|_n, |yx_3|_n, |yx_4|_n$ is less than $\frac{n}{2}$, we conclude that  $|yx_2|_n>\frac{n}{2}$ for all these five $y$. As above, we have $(5-1)\alpha < \frac{n}{2}$,  giving a contradiction again.


Next we may assume there  exist exactly two elements $y=1+t\alpha\in U(n)$ such that $|yx_3|_n>\frac{n}{2}$ and $ |yx_4|_n>\frac{n}{2}$. Thus exactly four $|yx_3|_n >\frac{n}{2}$ and exactly four $|y'x_4|_n>\frac{n}{2}$. A similar discussion on $x_2$ and  $x_3$ shows that exactly four $|y''x_2|_n >\frac{n}{2}$.

Since $|yx_1|_n=x_1$ for any $y=1+t\alpha\in U(n)$ ($t\in[0,6]$),  we have
\beqs M&=&\sum_{y=1+t\alpha\in U(n)\atop t\in[0,6]}\sum_{i=1}^4|yx_i|_n=\sum_{i=1}^4\sum_{y=1+t\alpha\in U(n)\atop t\in[0,6]} |yx_i|_n\\
&\geq& 6\times\frac{n-7}{2}\\&&
+(x_2'+(x_2'+\alpha)+(x_2'+3\alpha)+(x_2'+4\alpha)+(x_2'+5\alpha)+(x_2'+6\alpha))\\
&&+(x_3'+(x_3'+\alpha)+(x_3'+3\alpha)+(x_3'+4\alpha)+(x_3'+5\alpha)+(x_3'+6\alpha))\\
&&+(x_4'+(x_4'+\alpha)+(x_4'+3\alpha)+(x_4'+4\alpha)+(x_4'+5\alpha)+(x_4'+6\alpha))\\
&=&3n-21+6x_2'+6x_3'+6x_4'+57\alpha,\eeqs
where $|yx_i|_n=x_i'+ t_i\alpha$ and $x_i'<\alpha$.

Since there are exactly four $y$ such that $|yx_i|_n > \frac{n}{2}$ for  $i\in[2,4]$, we conclude that $x_i'+3\alpha>\frac{n}{2}$, which implies that $x_i'>\frac{\alpha}{2}$ for $i\in[2,4]$.  Now we infer that
\beqs M >3n-21+66\alpha=12n+3(\alpha-7)>12n,\eeqs
and thus there exists at least one $y=1+t\alpha$ such that $|yx_1|_n+|yx_2|_n+|yx_3|_n+|yx_4|_n=3n$. By Lemma~\ref{L2.1}, we get  $\ind(S)=1$ as desired.
\end{proof}

\begin{lemma}\label{L2.6}
If $f(x_1)=5$, $\gcd(f(x_1),f(x_2))=\gcd(f(x_1),f(x_3))=\gcd(f(x_1),f(x_4))=1$ and $5^2\not|n$, then $\ind(S)=1$.
\end{lemma}
\begin{proof}
The proof is similar to that of the above lemma.
\end{proof}

\begin{lemma}\label{L2.7}
If $\gcd\big(f(x_1),f(x_2)\big)=d>1$, then $\ind(S)=1$.
\end{lemma}
\begin{proof}
If $f(x_1)=f(x_2)=d$, the result follows from Lemma~\ref{L2.4}. So we may assume that $x_1=f(x_1)>d$. Note that $x_1=f(x_1)<\frac{n}{2}$.

Since $x_1>d$, there must exist a prime $p$ and a non-negative integer $s$ such that $p^s|x_2$, $p^{s+1}\not|x_2$ and  $p^{s+1}|x_1$ (in fact, we may choose $p$ to be any prime factor of $\frac{x_1}{d}$). Let $\alpha=\frac{n}{p^{s+1}}$. By Corollary~\ref{C2.3}, there exists $y=1+k\alpha\in U(n)$ such that $|yx_2|_n<\frac{n}{2}$. We note that $|yx_1|_n=x_1<\frac{n}{2}$.

By multiplying $S$ by such  $y$,  we may assume that $x_1<\frac{n}{2}$ and $x_2<\frac{n}{2}$.  Choose a prime $p$ such that $p|d$ and let $\alpha'=\frac{n}{p}$. Since $\gcd(d,x_3)=1$, $ \gcd(p, x_3)=1$, so it follows from Lemma~\ref{L2.2} that there exists $y_1=1+k_1\alpha'\in U(n)$ such that $|y_1x_3|_n<\frac{n}2$. Since $y_1$ fixes both $x_1$ and $x_2$,  it follows from  Lemma~\ref{L2.1} that $\ind(S)=1$.
\end{proof}

\begin{lemma}\label{L2.8}
If $f(x_1)>1$, $f(x_2)>1$ and $\gcd\big(f(x_1),f(x_2)\big)=1$, then $\ind(S)=1$.
\end{lemma}
\begin{proof}
First we assume that $x_1=f(x_1)<\frac{n}{2}$. Let $p$ and $q$ be the largest primes such that $p|f(x_1)$ and $q|f(x_2)$,  and  set $\alpha=\frac{n}{p}$. Without loss of generality, we may assume that $p > q$. In view of Lemma~\ref{L2.7}, we may also assume that $\gcd(f(x_1), f(x_i))=1$ for all $i\in [2, 4]$.

 Next, since $\gcd(x_1,q)=1$, we may assume that $x_3=w_1x_1+v_1q$ and $x_4=w_2x_1+v_2q$ where $\gcd(x_1, v_i)=1$ for all $i \in [1,2]$. As in Lemma~\ref{L2.2}, there exists at most one $t \in [0, p-1]$ such that $y=1+t\alpha\not\in U(n)$. If $(1+t\alpha)x_3=(1+s\alpha)x_3(\mod\;n)$, then $n|(t-s)\alpha v_1q$,  and thus $p|(t-s)$ (as $\gcd(p, v_1q)=1$), so $t=s$. A similar result holds for $x_4$.

If there doesn't exist any $y$ such that $|yx_3|_n<\frac{n}{2}$ and $|yx_4|_n<\frac{n}{2}$ and there exist at least three $y$ such that
both $|yx_3|_n>\frac{n}2$ and $|yx_4|_n>\frac{n}{2}$, then there exist at least $\frac{p-1}{2}+2$ many $y$ such that $|yx_3|_n>\frac{n}{2}$ or $|yx_4|_n>\frac{n}{2}$. This implies that $\frac{p}{2}>\frac{p-1}{2}+2-1=\frac{p+1}{2}$, giving a contradiction. Thus, either we can find $y=1+t\alpha\in U(n)$ such that $|yx_3|_n<\frac{n}{2}$ and $ |yx_4|_n<\frac{n}{2}$, or there exist at least $p-3$ many $y=1+t\alpha \in U(n)$ such that $|yx_3|_n$ and $|yx_4|_n$ stand in different sides of $\frac{n}{2}$ for each $y$. For the former case, as before we have $\ind(S)=1$ by Lemma~\ref{L2.1}.

Next we consider the latter case. If $p\geq 11$,  we can find $y=1+t\alpha\in U(n)$ such that $|yx_2|_n<\frac{n}{2}$.
For otherwise, for these $p-3$ many $y$ we have $|yx_2|_n>\frac{n}{2}$. As before, we infer that $\frac{p}{2}>p-4$ and thus $p<8$, giving a contradiction.

Now assume that $p=7$. Since $\gcd\big(f(x_1),f(x_2)\big)=1$, we conclude that $f(x_1)=7^\lambda$ and $f(x_2)=5^\mu$. If $7^2|n$,  then either we can find $y=1+t\alpha\in U(n)$ such that $|yx_3|_n<\frac{n}{2}$ and $ |yx_4|_n<\frac{n}{2}$, or, as before, there exist at least $6$ elements $y=1+t\alpha\in U(n)$ such that $|yx_3|_n$ and $|yx_4|_n$ stand in different sides of $\frac{n}{2}$. For the latter case, we can find $y \in U(n)$ such that at least two of  $|yx_2|_n<\frac{n}{2}$, $|yx_3|_n<\frac{n}{2}$ and $ |yx_4|_n<\frac{n}{2}$ hold. Thus in both cases we have $\ind(S)=1$ by Lemma~\ref{L2.1}.  Finally, if $7^2\not|n$, by Lemma~\ref{L2.5}, we have $\ind(S)=1$. This completes the proof.
\end{proof}
\begin{lemma} \label{L2.9}
If $f(x_1)=d>1$ and $f(x_2)=f(x_3)=f(x_4)=1$ (i.e. $x_2, x_3, x_4$ are co-prime to $n$), then $\ind(S)=1$.
\end{lemma}
\begin{proof}
Let $p$ be the largest prime factor of $f(x_1)$ and $\alpha =\frac{n}{p}$. Since $x_2,x_3,x_4$ are co-prime to $n$ (hence they are co-prime to $p$), we may assume that
$x_i=w_ip+v_i$ for $i\in[2,4]$,  where $v_i\in[1,p-1]$. Again, we can show that  $(1+t\alpha)x_i=(1+s\alpha)x_i(\mod\;n)$ for any $i\in [2, 4]$ if and only if $t=s$.

If $p\geq 11$ or $p^2|n$, a proof similar to that of Lemma~\ref{L2.8} shows that $\ind(S)=1$. If $p\leq 7$, $f(x_1)=p\in\{5,7\}$ and $p^2\not|n$, by Lemmas~\ref{L2.5} and \ref{L2.6},  we get $\ind(S)=1$ as desired. Finally, we consider the last case when $p=7$, $p^2 \not|n$ and $f(x_1)=5\cdot7=35$. Since $n$ has at least three different prime factors and $\alpha=\frac{n}{7}\geq 55$. As in the proof of Lemma~\ref{L2.5}, we may assume that $x_1=\frac{n-35}{2}$ and we can reduce to the only case that there are exactly four $y=1+t\alpha \in U(n)$ such that $|yx_i|_n > \frac{n}{2}$ for each $i \in [2,4]$. As before, we can estimate the sum $M$ as follows:

\beqs M&=&\sum_{y=1+t\alpha\in U(n)\atop t\in[0,6]}\sum_{i=1}^4|yx_i|_n=\sum_{i=1}^4\sum_{y=1+t\alpha\in U(n)\atop t\in[0,6]} |yx_i|_n\\
& > & 3n-105+66\alpha=12n+3(\alpha-35)> 12n.\eeqs
Thus there exists at least one $y=1+t\alpha \in U(n)$ such that $|yx_1|_n+|yx_2|_n+|yx_3|_n+|yx_4|_n=3n$. By Lemma~\ref{L2.1}, we get  $\ind(S)=1$ as desired.
\end{proof}

\section{Proof of Main Result}

As mentioned early, in [18] the authors settled the remaining case when $|G|$ is a product of two prime powers. However, the proof is quite long.  By applying an innovative method developed in this paper, we are able to give a new and very short proof for the above mentioned case. This  together with [9]  provides a complete solution to the index conjecture for the product of two prime-power case.

\begin{theo}\label{T3.1}
Let $G=\langle g\rangle$ be a finite cyclic group of order $|G|=n$ such that $\gcd(n,6)=1$ and $n=p^{\beta}q^{\gamma} $ is a product of two different prime powers. If $S=(x_1g)(x_2g)(x_3g)(x_4g)$ is any  minimal zero-sum  sequence over $G$, then  $\ind(S)=1$.
\end{theo}

\begin{proof}
In view of \cite[Theorem 1.3]{LP2}, we may assume  $f(x_i) >1$ for each $i\in [1,4]$. We may also assume that $\gcd\big(f(x_1), f(x_2),f(x_3), f(x_4)\big)=1$,    $p| \gcd\big(f(x_1), f(x_2)\big)$ and   \\ $q| \gcd\big(f(x_3), f(x_4)\big)$. Thus we get that $f(x_1)=p^{s_1}, f(x_2)=p^{s_2}, f(x_3)=q^{s_3}$, and $f(x_4)=q^{s_4}$ with $s_i \geq 1, i\in [1,4]$. Without loss of generality, we may assume that $x_1=f(x_1) < \frac{n}{2}$ and $ f(x_1)\geq f(x_2)$ (i.e. $s_1\geq s_2$). We divide the proof into two cases.

{\bf Case 1}. $f(x_1)=f(x_2)=\gcd\big(f(x_1),f(x_2)\big) >1$. As in Lemma~\ref{L2.4}, we can find $u\in U(n)$ such that $|ux_1|_n <\frac{n}{2}$ and $|ux_2|_n <\frac{n}{2}$.  Since $\gcd(|ux_3|_n, p)=1$, by Lemma~\ref{L2.2}, there exists $y=1+t\frac{n}{p}\in U(n)$ such that $|yux_3|_n< \frac{n}{2}$. Note also that $|yux_i|_n=|ux_i|_n < \frac{n}{2}$ for all $i\in [1, 2]$. So it follows from Lemma~\ref{L2.1} that $\ind(S)=1$.

{\bf Case 2}. $f(x_1)>f(x_2)=p^{s_2}$. Note that $p^{s_2}|f(x_2), p^{s_2+1} \not| f(x_2)$ and $p^{s_2+1} | f(x_1)$. By Corollary~\ref{C2.3}, there exist $u=1+t\alpha \in U(n)$ with $\alpha=\frac{n}{p^{s_2+1}}$ such that $|ux_2|_n < \frac{n}{2}$. Note also that $|ux_1|_n=x_1 < \frac{n}{2}$. As in  Case 1, we can find $y=1+t\frac{n}{p}\in U(n)$ such that  $|yux_i|_n< \frac{n}{2}$ for all $i\in [1,3]$. Therefore, $\ind(S)=1$ as desired.
\end{proof}


{\bf Proof of Theorem~\ref{mainthm}.}

If $n$ has at most two distinct prime factors, the result follows immediately from \cite{LP2} and Theorem 3.1. So we need only consider the case when $n$ has at least three distinct prime factors. Assume that $x_1=f(x_1)=d>1$ and  $n$ has at least three distinct prime factors. We divide the proof into the following two  cases:

{\bf Case 1}. At least one $\gcd\big(f(x_1),f(x_i)\big)>1$ for $i\in [2, 4]$. Without loss of generality, we may assume that $\gcd\big(f(x_1),f(x_2)\big)>1$. It follows from Lemma~\ref{L2.7} that $\ind(S)=1$.

{\bf Case 2}. All $\gcd\big(f(x_1),f(x_i)\big)=1$ for $i\in [2, 4]$. We divide the proof into two subcases.

{\bf Subcase 2.1}. At least one $f(x_i) >1$ for  $i\in [2, 4]$. Without loss of generality, we may assume that $f(x_2)>1$. The result follows from Lemma~\ref{L2.8}.

{\bf Subcase 2.2}. $f(x_2)=f(x_3)=f(x_4)=1$. The result follows from Lemma~\ref{L2.9}.   \hfill $\Box$

\vskip .5 cm
\begin{center}
{ACKNOWLEDGEMENTS}
\end{center}
We would like to thank the referee for valuable suggestions which helped in improving the readability of the paper. A part of research was carried out during a visit by the third author to
Jiangsu University. He would like to gratefully acknowledge the kind hospitality from the host
institution. This research was supported in part by the  NNSF of China (Grant No. 11001110, 11271131) and a Discovery Grant from the Natural Science and Engineering Research Council of Canada.

\vskip30pt
\def\refname{

\end{document}